\documentclass[11pt]{article}
\usepackage{graphicx}
\usepackage{amsmath}
\usepackage{amssymb}
\usepackage{amssymb,amsmath,latexsym}

\usepackage{amsthm,amsopn}

\def\cs{{\cal S}}
\def\ZZ{\mathbb{Z}}
\def\CC{\mathbb{C}}
\def\wt{\widetilde}

\newtheorem{thm}{Theorem}[section]

\newtheorem{corollary}[thm]{Corollary}
\newtheorem{example}[thm]{Example}

\newtheorem{conjecture}[thm]{Conjecture}
\newtheorem{theorem}[thm]{Theorem}

\renewcommand{\Re}{\mathbb{R}}

\title{Markov processes with product-form stationary distribution}

\author{
{\bf Krzysztof Burdzy\footnote{Partially supported by NSF Grant DMS-0600206.}
} \ and  \ {\bf David White} }

\date{\ }                

\begin{document}
\maketitle

\section{Introduction}

This research has been inspired by several papers on processes
with inert drift \cite{W, BW, BHP, BBCH}. The model involves a
``particle'' $X$ and an ``inert drift'' $L$, neither of which
is a Markov process by itself, but the vector process $(X,L)$
is Markov. It turns out that for some processes $(X,L)$, the
stationary measure has the product form; see \cite{BBCH}. The
first goal of this note is to give an explicit characterization
of all processes $(X,L)$ with a finite state space for $X$  and
a product form stationary distribution---see Theorem
\ref{thm:main}.

The second, more philosophical, goal of this paper is to
develop a simple tool that could help generate conjectures
about stationary distributions for processes with {\it
continuous} state space and inert drift. So far, the only paper
containing a rigorous result about the stationary distribution
for a process with continuous state space and inert drift,
\cite{BBCH}, was inspired by computer simulations. Examples
presented in Section \ref{sec:ex} lead to a variety of
conjectures that would be hard to arrive at using pure
intuition or computer simulations.

\section{The model}
Let $\cs = \{1, 2, \dots, N\}$ for some integer $N>1$ and let $d
\geq 1$ be an integer. We define a continuous time Markov process
$(X(t),L(t))$ on $\cs \times \Re^d$ as follows. We associate with
each state $j \in \cs$ a vector $v_j\in\Re^d, 1\leq j\leq N$.
Define $L_j(t)=\mu(\{s\in [0,t]: X(s)=j\})$, where $\mu$ is
Lebesgue measure, and let $L(t)=\sum_{j \in \cs} v_j L_j(t)$. To
make the ``reinforcement'' non-trivial, we assume that at least
one of $v_j$'s is not 0. Since $L$ will always belong to the
hyperplane spanned by $v_j$'s, we also assume that
$d=\dim(\mathop{\rm span}\{v_1, \ldots, v_N\})$.

We also select non-negative functions $a_{ij}(l)$ which define
the Poisson rates of jumps from state $i$ to $j$.  The rates
depend on $l = L(t)$. We assume that $a_{ij}$'s are
right-continuous with left limits. Formally speaking, the
process $(X,L)$ is defined by its generator $A$ as follows,
\[ Af(j,l)=v_j\cdot\nabla_l f(j,l)
+\sum_{i\neq j}a_{ji}(l)[f(i,l)-f(j,l)], \quad j=1,\dots, N,\ l
\in \Re^d, \] for $f: \{1,\dots,N\}\times\Re^d\rightarrow \Re$.

We assume that $(X,L)$ is irreducible in the sense of Harris,
i.e., for some open set $U\subset \Re^d$ and some $j_0\in \cs$,
for all $(x,l) \in \cs \times \Re^d$, we have for some $t>0$,
 $$P((X(t), L(t)) \in \{j_0\} \times U) >0.
 $$
We are interested only in processes satisfying (\ref{kb17})
below. Using that condition, it is easy to check Harris
irreducibility for each of our models by a direct argument. A
standard coupling argument shows that Harris irreducibility
implies uniqueness of the stationary probability distribution
(assuming existence of such).

The (formal) adjoint of $A$ is given by
 \begin{equation}\label{kb1}
 A^*g(j,l)=-v_j\cdot\nabla_l g(j,l) +\sum_{i\neq j}
 [a_{ij}(l)g(i,l)-a_{ji}(l)g(j,l)],\quad j=1,\dots, N,\ l \in
 \Re^d.
 \end{equation}
We are interested in invariant measures of product form so suppose
that $g(j,l)=p_j g(l)$, where $\sum_{j\in\cs} p_j=1$ and
$\int_{\Re^d} g(l) dl = 1$.  We may assume that $p_j>0$ for all
$j$; otherwise some points in $\cs$ are never visited. Under these
assumptions, (\ref{kb1}) becomes
 \begin{equation*}
 A^*g(j,l)= -p_j v_j\cdot\nabla g(l) +\sum_{i\neq j}
 [p_i a_{ij}(l)g(l)- p_j a_{ji}(l)g(l)],\quad j=1,\dots, N,\ l \in
 \Re^d.
 \end{equation*}

\begin{theorem}\label{thm:main}
Assume that for every $i$ and $j$, the function $l\to
a_{ij}(l)$ is continuous. A probability measure $p_j g(l)djdl$
is invariant for the process $(X,L)$ if and only if
 \begin{equation}\label{kb2}
 -p_j v_j\cdot\nabla g(l) +\sum_{i\neq j}
 [ p_i a_{ij}(l)g(l)- p_j a_{ji}(l)g(l)] = 0,
 \quad j=1,\dots, N,\ l \in \Re^d.
 \end{equation}
\end{theorem}

\begin{proof} Recall that the state space $\cal S$ for $X$ is
finite. Hence $v_* := \sup_{j\in\cal S} |v_j| < \infty$. Fix
arbitrary $r,t_* \in (0,\infty)$. It follows that,
 \begin{equation*}
 \sup_{i,j \in {\cal S}, l \in B(0, r+ 2 t_* v_*)} a_{ij} (l) = a_* < \infty.
 \end{equation*}
Note that we always have $|L(t) - L(u)| \leq v_* |t-u|$. Hence,
if $|L(0)| \leq r + t_* v_*$ and $s,t>0$, $s+t \leq t_*$, then
$|L(s+t)| \leq r +2 t_* v_*$ and, therefore,
 \begin{equation*}
 \sup_{j \in {\cal S}, u \leq s+t} a_{X(u),j} (L(u)) \leq a_* < \infty.
 \end{equation*}
This implies that the probability of two or more jumps on the
interval $[s,s+t]$ is $o(t)$. Assume that $|l| \leq r + t_*
v_*$ and $t \leq t_*$. Then we have the following three
estimates. First,
 \begin{equation}\label{eq:sep9-1}
P( X(t) = j \mid X(0) = i, L(0) = l) = a_{ij}(l) t  + R^1_{i,j,l}(t),
 \end{equation}
where the remainder $R^1_{i,j,l}(t)$ satisfies $\sup_{i,j
\in{\cal S}, l \in B(0, t_* v_*)} |R^1_{i,j,l}(t)| \leq R^1(t)$
for some $R^1(t)$ such that $\lim_{t\to 0} R^1(t)/t=0$.

Let $a_{ii} (l) = - \sum_{j\ne i} a_{ij}(l)$. We have
\begin{equation}\label{eq:sep9-2}
P( X(t) = i, L(t) = l + t v_i \mid X(0) = i, L(0) = l) = 1+
a_{ii}(l) t +  R^2_{i,l}(t),
 \end{equation}
where the remainder $R^2_{i,l}(t)$ satisfies $\sup_{i \in{\cal
S}, l \in B(0, t_* v_*)} |R^2_{i,l}(t)| \leq R^2(t)$ for some
$R^2(t)$ such that $\lim_{t\to 0} R^2(t)/t=0$.

Finally,
\begin{equation}\label{eq:sep9-3}
P( X(t) = i, L(t) \ne l + t v_i \mid X(0) = i, L(0) = l) =  R^3_{i,l}(t),
 \end{equation}
where the remainder $R^3_{i,l}(t)$ satisfies $\sup_{i \in{\cal
S}, l \in B(0, t_* v_*)} |R^3_{i,l}(t)| \leq R^3(t)$ for some
$R^3(t)$ such that $\lim_{t\to 0} R^3(t)/t=0$.

Now consider any $C^1$ function $f(j,l)$ with support in ${\cal
S} \times B(0, r)$. Recall that $|L(t) - L(u)| \leq v_* |t-u|$.
Hence, $E_{i,l} f(X_{t }, L_{t }) = 0 $ for $t\leq t_*$ and
$|l| \geq r + v_* t_*$.

Suppose that $|l_0| \leq r + v_* t_*$, $t_1 \in (0, t_*)$ and
$s \in (0, t_*-t_1)$. Then
\begin{align*}
& E_{i,l_0}  f(X_{t_1 + s}, L_{t_1 + s}) - E_{i,l} f(X_{t_1 }, L_{t_1 }) \\
&= \sum_{j \in \cal S} \int_{\Re^d} \sum_{k \in \cal S} \int_{\Re^d}
f(k, r) P( X(t_1+s) = k , L(t_1+s) \in dr \mid X(t_1) = j, L(t_1) = l)\\
& \qquad \times P( X(t_1) = j , L(t_1) \in dl \mid X(0) = i, L(0) = l_0)\\
&- \sum_{j \in \cal S} \int_{\Re^d}
f(j, l) P( X(t_1) = j , L(t_1) \in dr \mid X(0) = i, L(0) = l_0).
\end{align*}
We combine this formula with
\eqref{eq:sep9-1}-\eqref{eq:sep9-3} to see that,
\begin{align*}
& E_{i,l_0}  f(X_{t_1 + s}, L_{t_1 + s}) - E_{i,l} f(X_{t_1 },
L_{t_1 })\\
 & = \sum_{j \in \cal S} \int_{\Re^d} \sum_{k \in {\cal
S}, k\ne j} (f(k, l) + O(s))
(a_{jk}(l) s  + R^1_{j,k,l}(s))\\
& \qquad \times P( X(t_1) = j , L(t_1) \in dl \mid X(0) = i, L(0) = l_0)\\
&+ \sum_{j \in \cal S} \int_{\Re^d}
f(j, l + s v_j)
(1+ a_{jj}(l) s +  R^2_{j,l}(s))\\
& \qquad \times P( X(t_1) = j , L(t_1) \in dl \mid X(0) = i, L(0) = l_0)\\
&+ \sum_{j \in \cal S} \int_{\Re^d}
(f(j, l) + O(s))  R^3_{j,l}(s)\\
& \qquad \times P( X(t_1) = j , L(t_1) \in dl \mid X(0) = i, L(0) = l_0)\\
&- \sum_{j \in \cal S} \int_{\Re^d}
f(j, l) P( X(t_1) = j , L(t_1) \in dl \mid X(0) = i, L(0) = l_0),\\
\intertext{which can be rewritten as}
 & \sum_{j \in \cal S}
\int_{\Re^d} (f(j, l + s v_j) -f(j, l))
P( X(t_1) = j , L(t_1) \in dl \mid X(0) = i, L(0) = l_0)\\
& + \sum_{j \in \cal S} \int_{\Re^d} \Bigg(
f(j, l + s v_j) a_{jj}(l) + \sum_{k \in {\cal S}, k\ne j}
f(k, l) a_{jk}(l)\Bigg) s \\
&\qquad \times P( X(t_1) = j , L(t_1) \in dl \mid X(0) = i, L(0) = l_0)\\
& + \sum_{j \in \cal S} \int_{\Re^d}
\Bigg(\Bigg(\sum_{k \in {\cal S}, k\ne j} f(k, l) R^1_{j,k,l}(s) + O(s)
(a_{jk}(l) s  + R^1_{j,k,l}(s))
\Bigg) \\
& \qquad + f(j, l + s v_j)  R^2_{j,l}(s) + (f(j,l)+O(s))  R^3_{j,l}(s)\Bigg)
P( X(t_1) = j , L(t_1) \in dl \mid X(0) = i, L(0) = l_0).
\end{align*}

We will analyze the limit
\begin{equation*}
\lim_{s\downarrow 0} \frac 1 s
( E_{i,l_0}  f(X_{t_1 + s}, L_{t_1 + s}) - E_{i,l_0} f(X_{t_1 }, L_{t_1 }) ).
\end{equation*}
Note that
\begin{align*}
& \lim_{s\downarrow 0} \frac 1 s
\sum_{j \in \cal S} \int_{\Re^d}
\Bigg(\Bigg(\sum_{k \in {\cal S}, k\ne j} f(k, l)
R^1_{j,k,l}(s) + O(s) (a_{jk}(l) s  + R^1_{j,k,l}(s))
\Bigg) \\
& \qquad + f(j, l + s v_j)  R^2_{j,l}(s) + (f(j,l)+O(s))
R^3_{j,l}(s)\Bigg) P( X(t_1) = j , L(t_1) \in dl \mid X(0) = i,
L(0) = l_0) \\
& \qquad= 0.
\end{align*}
We also have
\begin{align*}
& \lim_{s\downarrow 0} \frac 1 s
\sum_{j \in \cal S} \int_{\Re^d} (f(j, l + s v_j) -f(j, l)) P(
X(t_1) = j , L(t_1) \in dl \mid X(0) = i, L(0) =
l_0)\\
& = \sum_{j \in \cal S} \int_{\Re^d} \nabla_l f(j, l )\cdot v_j \,  P(
X(t_1) = j , L(t_1) \in dl \mid X(0) = i, L(0) =
l_0),
\end{align*}
and
\begin{align*}
& \lim_{s\downarrow 0} \frac 1 s
\sum_{j \in \cal S} \int_{\Re^d} \Bigg(
f(j, l + s v_j) a_{jj}(l) + \sum_{k \in {\cal S}, k\ne j}
f(k, l) a_{jk}(l)\Bigg)s \\
&\qquad \times P( X(t_1) = j , L(t_1) \in dl \mid X(0) = i, L(0) = l_0)\\
&=\sum_{j \in \cal S} \int_{\Re^d}
 \sum_{k \in {\cal S}}
f(k, l) a_{jk}(l)  P( X(t_1) = j , L(t_1) \in dl \mid X(0) = i, L(0) = l_0).
\end{align*}
This implies that
\begin{align}
& \frac d {dt} E_{i,l_0} f(X_{t }, L_{t }) \Big|_{t=t_1}
= \lim_{s\downarrow 0} \frac 1 s
( E_{i,l_0}  f(X_{t_1 + s}, L_{t_1 + s}) - E_{i,l_0} f(X_{t_1 }, L_{t_1 }) )
\nonumber\\
&=\sum_{j \in \cal S} \int_{\Re^d}
\left(\nabla_l f(j, l )\cdot v_j + \sum_{k \in {\cal S}}
f(k, l) a_{jk}(l) \right) P( X(t_1) = j , L(t_1) \in dl \mid X(0) = i, L(0) = l_0)
\label{eq:s10-1}\\
&=\sum_{j \in \cal S} \int_{\Re^d}
A f (j,l) P( X(t_1) = j , L(t_1) \in dl \mid X(0) =
i, L(0) = l_0)\nonumber \\
&= E_{i,l_0} Af (X_{t_1}, L_{t_1}),\label{eq:s10-3}
\end{align}
for all $i$ and $|l_0| \leq r + v_* t_*$.

We will argue that
\begin{align}\label{eq:s10-2}
P( (X(t_1) , L(t_1)) \in \,\cdot\, \mid X(0) = i, L(0) = l)
\to
P( (X(t_1) , L(t_1)) \in \,\cdot\, \mid
X(0) = i, L(0) = l_0),
\end{align}
weakly when $l \to l_0$. Let $T_k$ be the time of the $k$-th
jump of $X$. We have
\begin{align*}
P( T_1 > t \mid X(0) = i, L(0) = l)
= \exp\left( \int _0^t a_{ii}(l+s v_i)ds\right).
\end{align*}
Since $l \to a_{ii}(l)$ is continuous, we conclude that
\begin{align*}
P( T_1> t \mid X(0) = i, L(0) = l)
\to P( T_1> t \mid X(0) = i, L(0) = l_0),
\end{align*}
weakly as $l\to l_0$. This and continuity of $l \to a_{ij}(l)$
for every $j$ implies that
\begin{align*}
P( (X(T_1) , L(T_1)) \in \,\cdot\, \mid X(0) = i, L(0) = l)
\to
P( (X(T_1) , L(T_1)) \in \,\cdot\, \mid
X(0) = i, L(0) = l_0),
\end{align*}
weakly when $l \to l_0$. By the strong Markov property applied
at $T_k$'s, we obtain inductively that
\begin{align*}
P( (X(T_k) , L(T_k)) \in \,\cdot\, \mid X(0) = i, L(0) = l)
\to
P( (X(T_k) , L(T_k)) \in \,\cdot\, \mid
X(0) = i, L(0) = l_0),
\end{align*}
when $l \to l_0$, for every $k\geq 1$. This easily implies
\eqref{eq:s10-2}, because the number of jumps is stochastically
bounded on any finite interval.

Since $(j,l) \hookrightarrow \nabla_l f(j, l )\cdot v_j + \sum_{k
\in {\cal S}} f(k, l) a_{jk}(l)$ is a continuous function, it
follows from \eqref{eq:s10-1} and \eqref{eq:s10-2} that $l_0
\hookrightarrow \frac d {dt} E_{i,l_0} f(X_{t }, L_{t
})\Big|_{t=t_1}$ is continuous on the set $|l_0| \leq r + v_*
t_*$.

Recall that $E_{i,l_0} f(X_{t }, L_{t }) = 0 $ for $t\leq t_*$
and $|l_0| \geq r + v_* t_*$. Hence, $l_0\to \frac d {dt}
E_{i,l_0} f(X_{t }, L_{t })\Big|_{t=t_1}$ is continuous for all
$t_1\leq t_*$ and all values of $i$.

Fix some $t\leq t_*$ and let $u_t(j,l) = E_{j,l} f(X_{t }, L_{t
})$. We have just shown that for a fixed $t\leq t_*$ and any
$j$, the function $l\to u_t(j,l)$ is $C^1$. Hence we can apply
\eqref{eq:s10-3} with $f(j,l) = u_t(j,l)$ to obtain,
\begin{align}\label{eq:s10-7}
\frac d {dt} E_{j,l} f(X_{t }, L_{t })
&=\lim_{s \downarrow 0} \frac 1s
(E_{j,l} f(X_{t+s}, L_{t+s}) - E_{j,l} f(X_{t }, L_{t }))\\
&=\lim_{s \downarrow 0} \frac 1s (E_{j,l} u_t(X_s,L_s) - u_t(j,l))
\nonumber\\
&= (A u_t)(j,l).\nonumber
\end{align}

Since $\sup_{j,l}\left( \nabla_l f(j, l )\cdot v_j + \sum_{k
\in {\cal S}} f(k, l) a_{jk}(l)\right)<\infty$, formula
\eqref{eq:s10-1} shows that
\begin{align}\label{eq:s10-6}
\sup_{j,l,s\leq t_*}
\left(\frac d {dt} E_{j,l} f(X_{s }, L_{s }) \right) <\infty.
\end{align}

Now assume that (\ref{kb2}) is true and let $\pi(dj,dl) =p_j
g(l)djdl$. In view of \eqref{eq:s10-6}, we can change the order
of integration in the following calculation. For $0\leq t_1 <
t_2 \leq t_*$, using \eqref{eq:s10-7},
\begin{align}\label{eq:s11-1}
E_\pi f&(X(t_2), L(t_2)) - E_\pi f(X(t_1), L(t_1))\\
&= \sum_{j \in \cal S} \int_{\Re^d}
 E_{j,l} f(X_{t_2 }, L_{t_2 }) p_j g(l)dl
- \sum_{j \in \cal S} \int_{\Re^d}
 E_{j,l} f(X_{t_1 }, L_{t_1 }) p_j g(l)dl
 \nonumber\\
&= \sum_{j \in \cal S} \int_{\Re^d}
 \int_{t_1}^{t_2} \frac d{ds} E_{j,l} f(X_{s }, L_{s })ds p_j g(l)dl
 \nonumber\\
&= \int_{t_1}^{t_2}\sum_{j \in \cal S} \int_{\Re^d}
 \frac d{ds} E_{j,l} f(X_{s }, L_{s }) p_j g(l)dlds
 \nonumber\\
&= \int_{t_1}^{t_2}\sum_{j \in \cal S} \int_{\Re^d}
 (A u_s)(j,l) p_j g(l)dlds.\nonumber
\end{align}
Let $h(j,l) = p_j g(l)$. For a fixed $j$ and $s\leq t_*$, the
function $u_s(j,l)=0$ outside a compact set, so we can use
integration by parts to show that
\begin{align}\label{eq:s11-2}
\sum_{j \in \cal S} \int_{\Re^d}
 (A u_s)(j,l) p_j g(l)dl
= \sum_{j \in \cal S} \int_{\Re^d}
 u_s(j,l) (A^* h)(j,l)dl.
\end{align}
We combine this with the previous formula and the assumption
that $A^* h \equiv 0$ to see that
\begin{align*}
E_\pi f&(X(t_2), L(t_2)) - E_\pi f(X(t_1), L(t_1))
= \int_{t_1}^{t_2}\sum_{j \in \cal S} \int_{\Re^d}
 u_s(j,l) (A^* h)(j,l)dl ds = 0.
\end{align*}
It follows that $t\to E_\pi f(X(t), L(t))$ is constant for
every $C^1$ function $f(j,l)$ with compact support. This proves
that the distributions of $(X(t_1), L(t_1))$ and $(X(t_2),
L(t_2))$ are identical under $\pi$, for all $0\leq t_1 < t_2
\leq t_*$.

Conversely, assume that $\pi(dj,dl)=p_j g(l)djdl$ is invariant.
Then the left hand side of \eqref{eq:s11-1} is zero for all
$0\leq t_1 < t_2 \leq t_*$. This implies that
\begin{align*}
\sum_{j \in \cal S} \int_{\Re^d}
 (A u_s)(j,l) p_j g(l)dl =0
\end{align*}
for a set of $s$ that is dense on $[0,\infty)$. By
\eqref{eq:s11-2},
\begin{align}\label{eq:s11-3}
 \sum_{j \in \cal S} \int_{\Re^d}
 u_s(j,l) (A^* h)(j,l)dl =0
\end{align}
for a set of $s$ that is dense on $[0,\infty)$. Note that
$\lim_{s\downarrow 0} u_s(j,l) = f(j,l)$. Hence, the collection
of $C^1$ functions $u_s(j,l)$, obtained by taking arbitrary
$C^1$ functions $f(j,l)$ with compact support and positive
reals $s$ dense in $[0,\infty)$, is dense in the family of
$C^1$ functions with compact support. This and \eqref{eq:s11-3}
imply that $A^* h \equiv 0$, that is, \eqref{kb2} holds.
\end{proof}

\bigskip

\begin{corollary}
If a probability measure $p_j g(l)djdl$ is invariant for the
process $(X,L)$ then
\begin{equation}\label{kb17}
\sum_{j\in \cs} p_j v_j=0.
\end{equation}
\end{corollary}

\begin{proof}
Summing \eqref{kb2} over $j$, we obtain
\begin{equation}\label{eq:s11-5}
\sum_{j\in \cs} - p_j v_j \cdot \nabla g(l)=0,
\end{equation}
for all $l$. Since $g$ is integrable over $\Re^d$, it is
standard to show that there exist $l_1, l_2, \dots, l_d$ which
span $\Re^d$. Applying \eqref{eq:s11-5} to all $l_1,
l_2,\dots$, we obtain \eqref{kb17}.
\end{proof}

It will be convenient to use the following notation,
 \begin{equation}\label{kb6}
 b_{ij}(l)= p_i a_{ij}(l)- p_j a_{ji}(l).
 \end{equation}
Note that $b_{ij} = - b_{ji}$.
\bigskip

\begin{corollary}
A probability measure $p_j g(l)djdl$ is invariant for the
process $(X,L)$ and $g(l)$ is the Gaussian density
\begin{align}\label{eq:s13-1}
g(l)=(2\pi)^{-d/2}\exp(-|l|^2/2),
\end{align}
if and only if the following equivalent conditions hold,
 \begin{equation}\label{kb4}
 p_j v_j\cdot l +\sum_{i\neq j}
 [p_i a_{ij}(l)- p_j a_{ji}(l)] = 0,
 \quad j=1,\dots, N,\ l \in \Re^d,
 \end{equation}
 \begin{equation}\label{kb5}
  p_j v_j\cdot l +\sum_{i\neq j}
  b_{ij}(l) = 0,
 \quad j=1,\dots, N,\ l \in \Re^d.
 \end{equation}
\end{corollary}

\begin{proof}
If $g(l)$ is the Gaussian density then $\nabla g(l)=-l g(l)$
and \eqref{kb2} is equivalent to \eqref{kb4}. Conversely, if
\eqref{kb2} and \eqref{kb4} are satisfied then $\nabla g(l)=-l
g(l)$, so $g(l)$ must have the form \eqref{eq:s13-1}.
\end{proof}

In the rest of the paper we will consider only processes
satisfying \eqref{kb4}-\eqref{kb5}.

\begin{example}
{\rm

We now present some choices for $a_{ij}$'s. Recall the notation
$x^+= \max(x,0)$, $x^-=-\min(x,0)$, and the fact that $x^+ -
x^- =x$. Given $v_j$'s, $p_j$'s and $b_{ij}$'s which satisfy
(\ref{kb5}) and the condition $b_{ij} = - b_{ji}$, we may take
  \begin{equation}\label{kb8}
 a_{ij}(l) = \left(b_{ij}(l)\right)^+ /p_i.
  \end{equation}
 Then
 \[
 a_{ji}(l) = \left(b_{ji}(l)\right)^+ /p_j
 = \left(-b_{ij}(l)\right)^+ /p_j
 = \left(b_{ij}(l)\right)^- /p_j,
 \]
so
 \[
 p_i a_{ij}(l) - p_j a_{ji}(l) =
 \left(b_{ij}(l)\right)^+
 - \left(b_{ij}(l)\right)^-
 = b_{ij}(l),
 \]
as desired.

The above is a special case, in a sense, of the following.
Suppose that $p_j=p_i$ for all $i$ and $j$. Assume that $v_j$'s
and $b_{ij}$'s satisfy (\ref{kb5}) and the condition $b_{ij} =
- b_{ji}$. Fix some $c>0$ and let
 \begin{equation}\label{kb7}
 a_{ij}(l) = \frac { b_{ij}(l)
 \exp(c b_{ij}(l)) } {\exp (c b_{ij}(l)) - \exp (-c b_{ij}(l))}.
 \end{equation}
It is elementary to check that with this definition,
(\ref{kb6}) is satisfied for all $i$ and $j$, because
$b_{ij}(l) = - b_{ji}(l)$. The formula (\ref{kb7}) arose
naturally in \cite{W}. Note that (\ref{kb8}) (with all $p_i$'s
equal) is the limit of (\ref{kb7}) as $c\to \infty$.

 }
\end{example}

\section{Approximation of processes with continuous state space}
\label{sec:ex}

This section is contains examples of processes $(X,L)$ with
finite state space for $X$, and conjectures concerned with
processes with continuous state space. There are no proofs in
this section

First we will consider processes that resemble diffusions with
reflection. In these models, the ``inert drift'' is accumulated
only at the ``boundary'' of the domain.

We will now assume that elements of $\cs$ are points in a
Euclidean space $\Re^n$ with $n\leq N$. We denote them $\cs =\{
x_1, x_2, \dots, x_N\}$. In other words, by abuse of notation,
we switch from $j$ to $x_j$. We also take $v_j\in \Re^n$, i.e.
$d=n$. Moreover, we limit ourselves to functions $b_{ij}(l)$ of
the form $b_{ij} \cdot l$ for some vector $b_{ij} \in \Re^n$.
Then (\ref{kb5}) becomes
\begin{alignat}{5}
0       & +b_{12}  &&+ b_{13}  &&+ \cdots  &&+ b_{1N}
&&= -p_1 v_1 \notag\\
-b_{12}     & - 0
    &&+ b_{23} &&+ \cdots  &&+ b_{2n} &&= -p_2 v_2\label{eq:10}\\
& && &&\dots\notag\\
-b_{1N}     & - b_{2N}  &&-b_{3N}   &&- \cdots  &&-0
     &&= -p_N v_N\notag.
\end{alignat}

Consider any orthogonal transformation $\Lambda: \Re^n \to
\Re^n$.
If $\{b_{ij}, v_j, p_j\}$ satisfy \eqref{eq:10}
then so do $\{\Lambda b_{ij}, \Lambda v_j, p_j\}$.

Suppose that $a_{ij}(l)$ have the form $a_{ij} \cdot l$ for
some $a_{ij} \in \Re^n$.
If $\{a_{ij}, v_j, p_j\}$ satisfy \eqref{kb4} then so do $\{\Lambda
a_{ij}, \Lambda v_j, p_j\}$. Moreover, the process with parameters
$\{a_{ij}, v_j\}$ has the same transition probabilities as the
one with parameters $\{\Lambda a_{ij}, \Lambda v_j\}$.

\begin{example}
{\rm

Our first example is a reflected random walk on the interval
$[0,1]$. Let $x_j = (j-1)/(N-1)$ for $j=1, \dots, N$. We will
construct a process with all $p_j$'s equal to each other, i.e.,
$p_j = 1/N$. We will take $l \in \Re^1$, $v_1=\alpha$ and
$v_N=-\alpha$, for some $\alpha = \alpha(N) >0$, and all other
$v_j=0$, so that the ``inert drift'' $L$ changes only at the
endpoints of the interval. We also allow jumps only between
adjacent points, so $b_{ij} = 0$ for $|i-j|>1$. Then
(\ref{eq:10}) yields
\begin{align*}
b_{12} &= -\alpha /N \\
-b_{12}+b_{23} &=0 \\
 \cdots \\
-b_{(N-1)N} &= \alpha /N.
\end{align*}
Solving this, we obtain $b_{i(i+1)}=-\alpha /N$ for all $i$.

\newcommand{\lmr}{A_R}
\newcommand{\lml}{A_L}

We would like to find a family of semi-discrete models indexed
by $N$ that would converge to a continuous process with
product-form stationary distribution as $N\to \infty$. For
$1<i<N$, we set $a_{i(i+1)}(l)=\lmr(l,N)$ and
$a_{(i+1)i}(l)=\lml(l,N)$. We would like the random walk to
have variance of order 1 at time 1, for large $N$, so we need
\begin{equation}\label{kb11}
\lmr + \lml = N^2.
\end{equation}
Since $b_{i(i+1)}=-\alpha /N$ for all $i$, $\lmr$ and $\lml$
have to satisfy
\begin{equation}\label{kb10}
\lml - \lmr = \alpha l.
\end{equation}
When $l$ is of order 1, we would like to have drift of order 1
at time 1, so we take $\alpha =N$. Then (\ref{kb10}) becomes
\begin{equation}\label{kb12}
\lml - \lmr = N l.
\end{equation}
Solving (\ref{kb11}) and (\ref{kb12}) gives
\begin{align*}
\lml=\frac{N^2+ N l}{2}, && \lmr=\frac{N^2- N l}{2}.
\end{align*}

Unfortunately, $\lml$ and $\lmr$ given by the above formula can
take negative values---this is not allowed because $a_{ij}$'s
have to be positive. However, for every $N$, the stationary
distribution of $L$ is standard normal, so $l$ typically takes
values of order 1. We are interested in large $N$ so,
intuitively speaking, $\lmr$ and $\lml$ are not likely to take
negative values. To make this heuristics rigorous, we modify
the formulas for $\lmr$ and $\lml$ as follows,
\begin{align}\label{kb13}
\lml=\frac{N^2+ N l}{2} \lor 0 \lor N l, && \lmr=\frac{N^2- N
l}{2} \lor 0\lor (-N l).
\end{align}
Let $P_N$ denote the distribution of $(X,L)$ with the above
parameters. We conjecture that as $N\to\infty$, $P_N$ converge
to the distribution of reflected Brownian motion in $[0,1]$
with inert drift, as defined in \cite{W, BBCH}. The stationary
distribution for this continuous time process is the product of
the uniform measure in $[0,1]$ and the standard normal; see
\cite{BBCH}.

}
\end{example}

\begin{example}\label{rbm}
{\rm

This example is a semi-discrete approximation to reflected
Brownian motion in a bounded Euclidean subdomain of $\Re^n$,
with inert drift. In this example we proceed in the reversed
order, starting with $b_{ij}$'s and $a_{ij}$'s.

Consider an open bounded connected set $D \subset \Re^n$. Let
$K$ be a (large) integer and let $D_K = \ZZ^n/K \cap D$, i.e.,
$D_K$ is the subset of the square lattice with mesh $1/K$ that
is inside $D$. We assume that $D_K$ is connected, i.e., any
vertices in $D_K$ are connected by a path in $D_K$ consisting
of edges of length $1/K$. We take $\cs = D_K$ and $l\in \Re^n$.

We will consider nearest neighbor random walk, i.e., we will
take $a_{ij}(l) = 0$ for $|x_i - x_j| > 1/K$. In analogy to
(\ref{kb13}), we define
\begin{equation}\label{kb14}
a_{ij}(l) = \frac{K^2}2 (1 + (x_i-x_j)\cdot l ) \lor 0 \lor K^2
(x_i-x_j)\cdot l.
\end{equation}
Then $b_{ij}(l) = (K^2/N) (x_i-x_j)\cdot l$. Let us call a point
in $\cs = D_K$ an interior point if it has $2n$ neighbors in
$D_K$. We now define $v_j$'s using (\ref{eq:10}) with
$p_j=1/|D_K|$. For all interior points $x_j$, the vector $v_j$ is
0, by symmetry. For all boundary (that is, non-interior) points
$x_j$, the vector $v_j$ is not 0.

Fix $D \subset \Re^n$ and consider large $K$. Let $P_K$ denote
the distribution of $(X,L)$ constructed in this example. We
conjecture that as $K\to\infty$, $P_K$ converge to the
distribution of normally reflected Brownian motion in $D$ with
inert drift, as defined in \cite{W, BBCH}. If $D$ is $C^2$ then
it is known that the stationary distribution for this
continuous time process is the product of the uniform measure
in $D$ and the standard Gaussian distribution; see \cite{BBCH}.

}
\end{example}

\bigskip

The next two examples are discrete counterparts of processes
with continuous state space and smooth inert drift. The setting
is similar to that in Example \ref{rbm}. We consider an open
bounded connected set $D \subset \Re^n$. Let $K$ be a (large)
integer and let $D_K = \ZZ^n/K \cap D$, i.e., $D_K$ is the
subset of the square lattice with mesh $1/K$ that is inside
$D$. We assume that $D_K$ is connected, i.e., any vertices in
$D_K$ are connected by a path in $D_K$ consisting of edges of
length $1/K$. We take $\cs = D_K$ and $l\in \Re^n$.

\begin{example}
{\rm

This example is concerned with a situation when the stationary
distribution has the form $p_j g(l)$ where $p_j$'s are not
necessarily equal. We start with a $C^2$ ``potential'' $V: D
\to \Re$. We will write $V_j$ instead of $V(x_j)$. Let $p_j = c
\exp (-V_j)$. We need an auxiliary function
$$
d_{ij} = \frac{ 2(p_i - p_j)}{p_i(V_j - V_i) - p_j (V_i-V_j) }.
$$
Note that $d_{ij} = d_{ji}$ and
for a fixed $i$, we have $d_{ij_K}\to 1$ when $K\to \infty$ and
$|i-j_K| = 1/K$.

Let $a_{ij}(l) = 0$ for $|x_i - x_j|> 1/K$, and for $|x_i -
x_j|= 1/K$,
\begin{equation*}
\wt a_{ij}(l) = \frac{ K^2}2
(2 + d_{ij}  (V_i - V_j) +  (x_j -x_i) \cdot l).
\end{equation*}
We set
\begin{equation}\label{kb18}
a_{ij}(l) =
\begin{cases}
\wt a_{ij}(l) \lor 0& \text{if $\wt a_{ji}(l) >0$,} \\
(K^2/2 p_i) (p_i+p_j)(x_j - x_i) \cdot l &
\text{otherwise.}
\end{cases}
\end{equation}
If $\wt a_{ji}(l) >0$ and $\wt a_{ij}(l) >0$ then
\begin{align*}
b_{ij}(l) &=
p_i a_{ij}(l) - p_j a_{ji}(l) \\
&= \frac{K^2}2
(2(p_i - p_j) +
(p_i(V_i - V_j) - p_j (V_j-V_i))
d_{ij}
+ (p_i(x_j - x_i) - p_j (x_i-x_j))\cdot l) \\
&= \frac{K^2}2
(2(p_i - p_j) - 2(p_i - p_j) + (p_i+p_j)(x_j - x_i)\cdot l)\\
&= \frac{K^2}2 (p_i+p_j)(x_j - x_i)\cdot l.
\end{align*}
It follows from (\ref{kb18}) that the above formula holds also
if $\wt a_{ji}(l) \leq 0$ or $\wt a_{ij}(l) \leq 0$. Consider
an interior point $x_j$. For (\ref{kb5}) to be satisfied, we
have to take
\begin{equation*}
v_j = -\frac1{p_j} \sum_{|x_i-x_j| = 1/K} \frac{K^2}2
(p_i+p_j)(x_j - x_i).
\end{equation*}
For large $K$, series expansion shows that
\begin{align*}
v_j \approx -\nabla V.
\end{align*}

Fix $D \subset \Re^n$ and consider large $K$. Let $P_K$ denote
the distribution of $(X,L)$ constructed in this example. We
recall the following SDE from \cite{BBCH},
\begin{align*}
dY_t &= -\nabla V(Y_t)\, dt + S_t\, dt + dB_t\;, \\
dS_t &= -\nabla V(Y_t)\, dt\;,
\end{align*}
where $B$ is standard $n$-dimensional Brownian motion and $V$
is as above. Let $P_*$ denote the distribution of $(Y,S)$. We
conjecture that as $K\to\infty$, $P_K$ converge to $P_*$. Under
mild assumptions on $V$, it is known that the stationary
distribution for $(Y,S)$ is the product of the measure
$\exp(-V(x))dx$ and the standard Gaussian distribution; see
\cite{BBCH}.

}
\end{example}

\begin{example}
{\rm

We again consider the situation when all $p_j$'s are equal, i.e.,
$p_j = 1/N$. Consider a $C^2$ function $V: D \to \Re$. We let
$a_{ij}(l) = 0$ for $|x_i - x_j| > 1/K$. If $|x_i - x_j| = 1/K$,
we let
\begin{equation*}
\wt a_{ij}(l) = \frac{ K^2}2(1+ (V_j + V_i)(x_j-x_i) \cdot l).
\end{equation*}
We set
\begin{equation}\label{kb19}
a_{ij}(l) =
\begin{cases}
\wt a_{ij}(l) \lor 0& \text{if $\wt a_{ji}(l) >0$,} \\
K^2 (V_j + V_i)(x_j-x_i) \cdot l & \text{otherwise.}
\end{cases}
\end{equation}
Then $b_{ij}(l) = (1/N)K^2 (V_j + V_i)(x_j-x_i) \cdot l $ and
\begin{equation*}
v_j =  K^2 \sum_{|x_i - x_j| = 1/K} (V_j + V_i)(x_j-x_i).
\end{equation*}
For large $K$, we have $v_j \approx - 2 \nabla V$.

Fix $D \subset \Re^n$ and consider large $K$. Let $P_K$ denote
the distribution of $(X,L)$ constructed in this example.
Consider the following SDE,
\begin{align*}
dY_t &= V(Y_t) S_t\, dt + dB_t\;, \\
dS_t &= -2\nabla V(Y_t)\, dt\;,
\end{align*}
where $B$ is standard $n$-dimensional Brownian motion and $V$
is as above. Let $P_*$ denote the distribution of $(Y,S)$. We
conjecture that as $K\to\infty$, $P_K$ converge to $P_*$, and
that the stationary distribution for $(Y,S)$ is the product of
the uniform measure on $D$ and the standard Gaussian
distribution.

}
\end{example}

\bigskip

The next example and conjecture are devoted to examples where
the inert drift is related to the curvature of the state space,
in a suitable sense.

\begin{example}
{\rm

In this example, we will identify $\Re^2$ and $\CC$. The
imaginary unit will be denoted by $i$, as usual. Let $\cs$
consist of $N$ points on a circle with radius $r>0$, $x_j=r
\exp(j2\pi i/N), j=1,\dots, N$. We assume that the $p_j$'s are
all equal to each other.

For any pair of adjacent points $x_j$ and $x_k$, we let
$$
\wt a_{jk}(l)  = \frac{N^2}2 (1 +(x_k-x_j)\cdot l),
$$
and
\begin{equation*}
a_{jk}(l) =
\begin{cases}
\wt a_{jk}(l) \lor 0& \text{if $\wt a_{kj}(l) >0$,} \\
N^2(x_{k}-x_j)\cdot l & \text{otherwise,}
\end{cases}
\end{equation*}
with the other $a_{kj}(l)=0$. Then $b_{j(j+1)}=N(x_{j+1}-x_j)\cdot
l$, and by \eqref{kb5} we have
\[ v_{j}=N^2(x_{j-1}-x_{j})+N^2(x_{j+1}-x_j)
= 2 N^2  ( \cos(2\pi/N)-1) x_j.
\]
Note that $v_j \to -4 \pi^2 x_j$ when $N\to \infty$.

Let $P_N$ be the distribution of $(X,L)$ constructed above.

Let $\cal C$ be the circle with radius $r>0$ and center 0, and
let $T_y$ be the projection of $\Re^2$ onto the tangent line to
$\cal C$ at $y \in \cal C$. Consider the following SDE,
\begin{align*}
dY_t &=  T_{Y_t}(S_t)\, dt + dB_t\;, \\
dS_t &= - 4\pi^2  Y_t \, dt\;,
\end{align*}
where $Y$ takes values in $\cal C$ and $B$ is Brownian motion
on this circle. Let $P_*$ be the distribution of $(Y,S)$. We
conjecture that as $N\to\infty$, $P_N$ converge to $P_*$, and
that the stationary distribution for $(Y,S)$ is the product of
the uniform measure on the circle and the standard Gaussian
distribution.

}
\end{example}

\begin{conjecture}
{\rm

We propose a generalization of the conjecture stated in the
previous example. We could start with an explicit discrete
approximation, just like in other examples discussed so far.
The notation would be complicated and the whole procedure would
not be illuminating, so we skip the approximation and discuss
only the continuous model.

Let ${\cal S} \subset \Re^n$ be a smooth $(n-1)$-dimensional
surface, let $T_y$ be the projection of $\Re^n$ onto the
tangent space to $\cal S$ at $y \in \cal S$, let ${\bf n}(y)$
be the inward normal to $\cal S$ at $y\in \cal S$, and let
$\rho$ be the mean curvature at $y \in \cal S$. Consider the
following SDE,
\begin{align*}
dY_t &=  T_{Y_t}(S_t)\, dt + dB_t\;, \\
dS_t &=  c_0 \rho^{-1} {\bf n}(Y_t) \, dt\;,
\end{align*}
where $Y$ takes values in $\cal S$ and $B$ is Brownian motion
on this surface. We conjecture that for some $c_0$ depending
only on the dimension $n$, the stationary distribution for
$(Y,S)$ exists, is unique and is the product of the uniform
measure on $\cal S$ and the standard Gaussian distribution.

}
\end{conjecture}

\bigskip

We end with examples of processes that are discrete
approximations of continuous-space processes with jumps. It is
not hard to construct examples of discrete-space processes that
converge in distribution to continuous-space processes with
jumps. Stable processes are a popular family of processes with
jumps. These and similar examples of processes with jumps allow
for jumps of arbitrary size, and this does not mesh well with
our model because we assume a finite state space for $X$. Jump
processes confined to a bounded domain have been defined (see,
e.g., \cite{BBC}) but their structure is not very simple. For
these technical reasons, we will present approximations to
processes similar to the stable process wrapped around a
circle.

In both examples, we will identify $\Re^2$ and $\CC$. Let $\cs$
consist of $N$ points on the unit circle $D$, $x_j= \exp(j2\pi
i/N), j=1,\dots, N$. We assume that the $p_j$'s are all equal
to each other, hence, $p_j = 1/N$. In these examples, $L$ takes
values in $\Re$, not $\Re^2$.

\begin{example}
{\rm

Consider a $C^3$-function $V: D \to \Re$. We write $V_j = V(x_j)$.
We define
\[A(j,k)=\begin{cases}
1 & \text{if $x_j$ and $x_ k$ are adjacent on the unit circle,} \\
0 & \text{otherwise}.
\end{cases}\]
For any pair of points $x_j$ and $x_k$, not
necessarily adjacent, we let
$$
\wt a_{jk}(l)  = \frac {N^2}2 (V_k - V_j)  A(j,k)
 l + \frac 1N \sum _{n\in \ZZ} |(k-j)+nN|^{-1-\alpha} ,
$$
where $\alpha \in (0,2)$. We define
\begin{equation*}
a_{jk}(l) =
\begin{cases}
\wt a_{jk}(l) \lor 0& \text{if $\wt a_{kj}(l) >0$,} \\
N^2 (V_k - V_j)  A(j,k) l & \text{otherwise.}
\end{cases}
\end{equation*}
Then
 $$b_{jk}(l)= N (V_k - V_j)  A(j,k) l
 $$
and by \eqref{kb5} we have
$$
 v_k= -N^2 \sum_ {j: A(k,j) = 1} V_k - V_j.
$$
Note that $v_k \to  \Delta V(x) =  V''(x) $ when $N\to \infty$ and
$x_k \to x$.

Let $P_N$ be the distribution of $(X,L)$ constructed above. Let
$W(x) = V(e^{ix})$ and let $(Z,S)$ be a Markov process with the
state space $\Re\times \Re$ and the following transition
probabilities. The component $Z$ is a jump process with the drift
$\nabla W(Z) S = W'(Z) S$. The jump density for the process $Z$ is
$ \sum_{n\in \ZZ} |(x-y) + n 2\pi|^{-1-\alpha}$. We let $S_t =
\int_0^t \Delta W(Z_s) ds$. Let $Y_t = \exp(iZ_t)$ and $P_*$ be
the distribution of $(Y,S)$. We conjecture that $P_N \to P_*$ as
$N\to \infty$ and the process $(Y,S)$ has the stationary
distribution which is the product of the uniform measure on $D$
and the standard normal distribution. The process $(Y,S)$ is a
``stable process with index $\alpha$, with inert drift, wrapped on
the unit circle.''

}
\end{example}

\begin{example}
{\rm

Consider a continuous function $V: D \to \Re$ with $\int_D V(x)
dx =0$. Recall the notation $V_j = V(x_j)$. For any pair of
points $x_j$ and $x_k$, not necessarily adjacent, we let
$$
\wt a_{jk}(l)  = \frac 1N \left( \frac 12 (V_k - V_j) l + \sum
_{n\in \ZZ} |(k-j)+nN|^{-1-\alpha}\right) ,
$$
where $\alpha \in (0,2)$. We define
\begin{equation*}
a_{jk}(l) =
\begin{cases}
\wt a_{jk}(l) \lor 0& \text{if $\wt a_{kj}(l) >0$,} \\
\frac 1N (V_k - V_j) l &
\text{otherwise.}
\end{cases}
\end{equation*}
Then $b_{jk}(l)= (1/N^2) (V_k - V_j)  l $ and by \eqref{kb5} we
have
$$
 v_k= \frac 1N \sum_{1\leq j \leq N, j\ne k} V_k - V_j.
$$
Note that if $\arg x_k \to y $ when $N\to \infty $ then $v_k \to
V(e^{iy}) - \int_D V(x)dx = V(e^{iy})$.

Let $P_N$ be the distribution of $(X,L)$ constructed above. Let
$W(x) = V(e^{ix})$ and let $(Z,S)$ be a Markov process with the
state space $\Re\times \Re$ and the following transition
probabilities. The component $Z$ is a jump process with the
jump density $ f(x) = (W(x) -W(y))s - \int \sum_{n\in \ZZ}
((x-y) + n 2\pi)^{-1-\alpha}$ at time $t$, given $\{Z_t = y,
S_t = s\}$. We let $S_t = \int_0^t W(Z_s) ds$. Let $Y_t =
\exp(i Z_t)$ and $P_*$ be the distribution of $(Y,S)$. We
conjecture that $P_N \to P_*$ as $N\to \infty$ and the process
$(Y,S)$ has the stationary distribution which is the product of
the uniform measure on $D$ and the standard normal
distribution.

}
\end{example}

\bigskip

\section{Acknowledgments}

We are grateful to Zhenqing Chen and Tadeusz Kulczycki for very
helpful suggestions.

 \vskip1truein

\bibliographystyle{amsplain}
\bibliography{rwdrift}

\vskip 0.6truein

\noindent{\bf Krzysztof Burdzy:}

Department of Mathematics, University of Washington,  Seattle,
WA 98195, USA.

 Email: \texttt{burdzy@math.washington.edu}

\bigskip

\noindent{\bf David White:}

 Department of Mathematics, Belmont University, Nashville TN 37212

  Email: \texttt{white@u.washington.edu}

\end{document}